\documentclass[preprint,12pt]{elsarticle}
\usepackage{amssymb}
\usepackage{amsthm}
\usepackage{amsmath}
\usepackage{epic}
\usepackage{graphicx}
\usepackage{subfig}
\usepackage{graphicx}
\usepackage{caption,color}   
\usepackage{amssymb}
\usepackage{amsthm}
\usepackage{amsmath}
\usepackage{epic}
\usepackage{setspace}
\usepackage{float}
\usepackage{multirow}

\usepackage{hyperref}
\usepackage{xcolor}
\usepackage{marginnote}
\usepackage{booktabs}
\usepackage{algorithm, algorithmicx, algpseudocode}
\newtheorem{thm}{Theorem}[section]

\newtheorem{lem}[thm]{Lemma}

\journal{~}

\begin{document}
\begin{spacing}{1.15}

\begin{frontmatter}

\title{\textbf{Perron-Frobenius theorem for dual tensors and its applications}}
\author{Changjiang Bu}
\author{Yue Chu}
\author{Qingying Zhang}
\author{Jiang Zhou}
\address{College of Mathematical Sciences, Harbin Engineering University, Harbin 150001, PR China}

\begin{abstract}
The Perron-Frobenius theorem of nonnegative matrices is a classical result on spectral theory of matrices, which has wide applications in many domains. In this paper, we give the Perron-Frobenius theorem for dual tensors, that is, a dual tensor with weakly irreducible nonnegative standard part has a positive dual eigenvalue with a positive dual eigenvector. We give an explicit formula for the dual part of the positive dual eigenvector by using generalized inverses of an $M$-matrix. By considering the natural correspondence between tensors (matrices) and hypergraphs (graphs), some basic properties on the positive dual eigenvalue and positive dual eigenvector of hypergraphs are obtained. As applications, we introduce dual centrality measures for vertices of graphs and hypergraphs.
By introducing a dual perturbation, vertices that are tied under eigenvector centrality can be effectively distinguished. In our numerical experiments, by perturbing specific structures, we successfully differentiated vertices in regular graphs and hypergraphs that were previously indistinguishable.
\end{abstract}


\end{frontmatter}


\section{Introduction}
Dual numbers were first introduced by Clifford in 1873 as a number system extending the real numbers through the introduction of a new element $\epsilon$ (the dual unit) with the property that $\epsilon\neq0$, ${\epsilon}^2=0$ and $\epsilon$ is commutative with real numbers \cite{clifford1871preliminary}. The field of real numbers and the set of dual numbers are denoted by $\mathbb{R}$ and $\mathbb{D}$, respectively. We can write a dual number $\lambda\in \mathbb{D}$ as $\lambda=\lambda_s+\lambda_d\epsilon$, where $\lambda_s,\lambda_d\in\mathbb{R}$ are referred to as the standard part and dual part, respectively. Similarly, a dual vector $x\in\mathbb{D}^n$ can be written as $x=x_s+x_d\epsilon$, where $x_s,x_d$ are real vectors of dimension $n$. A dual number $\lambda$ (resp. dual vector $x$) is called \textit{positive} if $\lambda_s>0$ (resp. $x_s$ is positive). Dual matrices are matrices whose entries are dual numbers. Dual numbers and dual matrices have found wide applications in areas such as robotics \cite{gu1987dual}, kinematics \cite{fischer2017dual,moreno2024automatic} and perturbation analysis \cite{jiang2024perturbation,qi2024eigenvalues}, etc.

The Perron-Frobenius theorem of nonnegative matrices \cite{meyer2023matrix} has wide applications in graph theory \cite{kirkland1996characteristic}, network science \cite{bihari2015eigenvector}, Google's PageRank search engine \cite{langville2005survey} and system theory \cite{aeyels2002extension} .etc. Qi and Cui \cite{qi2024dual} extended the Perron-Frobenius theory to dual number matrices with irreducible nonnegative standard parts, and use their theoretical results to study the perturbation analysis for the dual Markov chain and the perturbed Markov chain.

There exists a natural correspondence between nonnegative tensors (matrices) and hypergraphs (graphs).
The Perron-Frobenius theorem of nonnegative tensors has been established in \cite{chang2008perron,yang2010further,friedland2013perron}, which provide an important tool for the research on tensor eigenvalues of graphs and hypergraphs \cite{qi2017tensor,cooper2012spectra,xu2023two,liu2023generalization,chen2024spectra}.
The adjacency tensor of a connected hypergraph is a weakly irreducible nonnegative tensor, and it has a positive eigenvalue (spectral radius) with a positive eigenvector.

Centrality measures aim to identify the most important vertices in networks.
Over the years, researchers have proposed many centrality measures \cite{albert2000error,freeman1991centrality,alahakoon2011k,estrada2005subgraph,zhou2023estrada,lee2021betweenness,tudisco2018node}. Eigenvector centrality \cite{bonacich1972factoring} is an important application of the Perron-Frobenius theorem of nonnegative matrices, which uses a positive eigenvector of the graph's adjacency matrix as the centrality scores of vertices. In \cite{benson2019three}, Benson uses the positive eigenvector of the hypergraph's adjacency tensor as the centrality scores of vertices, and extended eigenvector centrality from graphs to hypergraphs.

Dual tensors are tensors whose entries are dual numbers. We will investigate the spectral properties of a dual tensor and its applications in the centrality of hypergraphs. The rest of the paper is organized as follows. In Section 2, some definitions, notations and auxiliary lemmas are introduced. In Section 3, we give the Perron-Frobenius theorem of dual tensors, which extends the work of Qi and Cui \cite{qi2024dual} to the tensor case. We show that a dual tensor with weakly irreducible nonnegative standard part has a positive dual eigenvalue with a positive dual eigenvector, and give an explicit formula for the dual part of the positive dual eigenvector by using generalized inverses of an $M$-matrix.

In Section 4, we derive some basic properties on the positive dual eigenvalue and positive dual eigenvector of uniform hypergraphs, and propose the dual centrality vector. In Section 5, we use dual eigenvalue and dual centrality vector to rank the centrality of vertices in graphs and hypergraphs.

\section{Preliminaries}

Let $R(A)$ denote the range of a matrix $A$. The \textit{$\{1\}$-inverse} of $A$ is a matrix $X$ such that $AXA=A$. For a square matrix $M$, the \textit{group inverse} of $M$, denoted by $M^\#$, is the unique matrix $X$ such that $MXM=M,~XMX=X$ and $MX=XM$. It is known that $M^\#$ exists if and only if $\mbox{\rm rank}(M)=\mbox{\rm rank}(M^2)$ (see \cite{ben2003generalized}).
\begin{lem}\label{lem2.1}
Let $M$ be a real square matrix such that $\mbox{\rm rank}(M)=\mbox{\rm rank}(M^2)$. If $x\in R(M)$ and $y\in R(M^\top)$, then
\begin{eqnarray*}
y^\top M^\#x=y^\top M^{(1)}x
\end{eqnarray*}
for any $\{1\}$-inverse $M^{(1)}$ of $M$.
\end{lem}
\begin{proof}
There are vectors $z_1,z_2$ such that $y=M^\top z_1$ and $x=Mz_2$. Hence
\begin{eqnarray*}
y^\top M^\#x=y^\top M^{(1)}x=z_1^\top Mz_2
\end{eqnarray*}
for any $\{1\}$-inverse $M^{(1)}$ of $M$.
\end{proof}
A real square matrix $A$ is called a \textit{Z-matrix} if $(A)_{ij}\leq0$ for any $i\neq j$. If a Z-matrix $A$ can be written as $A=sI-B$ such that $B$ is a nonnegative matrix and $s\geq\rho(B)$ ($\rho(B)$ is the spectral radius of $B$), then $A$ is called an \textit{M-matrix}.
\begin{lem}\label{lem2.2} \cite{berman1994nonnegative}
Let $A$ be an irreducible Z-matrix. If $Ax=0$ for a positive real vector $x$, then $A$ is an irreducible singular M-matrix.
\end{lem}

\begin{lem}\label{lem2.3} \cite{berman1994nonnegative}
Let $A$ be an irreducible singular M-matrix of order $n$. Then ${\rm rank}(A)=n-1$ and every $k\times k$ ($k<n$) principal submatrix of $A$ is a nonsingular M-matrix.
\end{lem}

An order $m$ dimension $n$ complex tensor $\mathcal{A}=(a_{i_{1}i_{2}...i_{m}})$ is a multidimensional array with $n^{m}$ entries, where $i_{j}=1,2,...,n$, $j=1,2,...,m$. Let $\mathbb{C}$ be the complex field and let $\mathbb{C}^{n}$ be the set of $n$-dimensional complex vectors. For a vector $x=(x_{1},\ldots,x_{n})^\top\in\mathbb{C}^{n}$, $\mathcal{A}x^{m-1}$ is a vector in $\mathbb{C}^{n}$ whose $i$-th component is $\sum_{i_{2},\ldots,i_{m}=1}^{n}a_{ii_{2}...i_{m}}x_{i_{2}}\cdots x_{i_{m}}$. If there exist $\lambda\in\mathbb{C}$ and nonzero vector $x=(x_{1},\ldots,x_{n})^\top\in\mathbb{C}^{n}$ such that
\begin{eqnarray*}
\mathcal{A}x^{m-1}=\lambda(x_1^{m-1},...,x_n^{m-1})^\top,
\end{eqnarray*}
then $\lambda$ is an \textit{eigenvalue} of $\mathcal{A}$ with an associated eigenvector $x$ (see \cite{qi2005eigenvalues,lim2005singular}).

For a tensor $\mathcal{A}=(a_{i_1\cdots i_m })\in\mathbb{C}^{n\times\cdots\times n}$, we associate with $\mathcal{A}$ a digraph $\Gamma_{\mathcal{A}}$ as follows. The vertex set of $\Gamma_{\mathcal{A}}$ is $V(\mathcal{A})=\left\{1,\ldots,n\right\}$, the arc set of $\Gamma_{\mathcal{A}}$ is $E(\mathcal{A }) = \{(i,j)|a_{ii_2\cdots i_m}\neq0,j\in\{i_2,\ldots,i_m\}\neq\{i,\ldots,i\}\}$. $\mathcal{A}$ is said to be \textit{weakly irreducible} if $\Gamma_{\mathcal{A}}$ is strongly connected
\begin{lem}\label{lem2.4} \cite{friedland2013perron}
Let $\mathcal{A}$ be a weakly irreducible nonnegative tensor. Then there exists positive real number $\lambda$ and positive real vector $x$ such that $\mathcal{A}x^{m-1}=\lambda x^{[m-1]}$ and $\lambda$ equals to the spectral radius of $\mathcal{A}$.
\end{lem}

\section{Eigenvalues of dual tensors}
Let $\mathcal{A}$ be a dual tensor of order $m$ and dimension $n$. Then $\mathcal{A}$ can be written as $\mathcal{A}=\mathcal{A}_s+\mathcal{A}_d\epsilon$, where $\mathcal{A}_s$ and $\mathcal{A}_d$ are real tensors representing the standard and dual parts, respectively. Let $\mathcal{I}$ denote a diagonal tensor such that every diagonal entry is $1$.

For a dual vector $x=(x_{1},\ldots,x_{n})^\top\in\mathbb{D}^{n}$, let $x^{[m-1]}=(x_1^{m-1},...,x_n^{m-1})^\top$, and let $\mathcal{A}x^{m-1}$ denote a dual vector in $\mathbb{D}^{n}$ whose $i$-th component is
\begin{eqnarray*}
(\mathcal{A}x^{m-1})_i=\sum_{i_{2},...,i_{m}=1}^{n}a_{ii_{2}...i_{m}}x_{i_{2}}\cdots x_{i_{m}}.
\end{eqnarray*}
If there exists dual number $\lambda$ and dual vector $x=x_s+x_d\epsilon$ ($x_s\neq0$) such that $\mathcal{A}x^{m-1}=\lambda x^{[m-1]}$, then we say that $\lambda$ is an \textit{eigenvalue} of $\mathcal{A}$ with an associated eigenvector $x$.
\begin{thm}
Let $\mathcal{A}=\mathcal{A}_s+\mathcal{A}_d\epsilon$ be a dual tensor of order $m$ and dimension $n$. For dual number $\lambda=\lambda_s+\lambda_d\epsilon$ and dual vector $x=x_s+x_d\epsilon$, we have $\mathcal{A}x^{m-1}=\lambda x^{[m-1]}$ if and only if $\mathcal{A}_sx_s^{m-1}=\lambda_sx_s^{[m-1]}$ and
\begin{eqnarray*}
Mx_d=(\mathcal{A}_d-\lambda_d\mathcal{I})x_s^{m-1},
\end{eqnarray*}
where $M=(m-1)\lambda_s({\rm diag}((x_s)_1,\ldots,(x_s)_n))^{m-2}-\sum_{k=2}^mA^{(k)}$, $A^{(k)}$ is an $n\times n$ matrix whose $(i,j)$-entry is ($1\leq i\leq j\leq n$)
\begin{eqnarray*}
(A^{(k)})_{ij}=\sum_{i_2,\ldots,i_{k-1},i_{k+1},\ldots,i_m=1}^n(\mathcal{A}_s)_{ii_2\cdots i_{k-1}ji_{k+1}\cdots i_m}(x_s)_{i_2}\cdots (x_s)_{i_{k-1}}(x_s)_{i_{k+1}}\cdots(x_s)_{i_m}.
\end{eqnarray*}
Moreover, if $\mathcal{A}_s$ is symmetry, then $M=(m-1)(\lambda_s({\rm diag}((x_s)_1,\ldots,(x_s)_n))^{m-2}-A)$, where $A$ is an $n\times n$ symmetry matrix with entries
\begin{eqnarray*}
(A)_{ij}=\sum_{i_3,i_4,\ldots,i_m=1}^n(\mathcal{A}_s)_{iji_3i_4\cdots i_m}(x_s)_{i_3}(x_s)_{i_4}\cdots(x_s)_{i_m}~(1\leq i\leq j\leq n).
\end{eqnarray*}
\end{thm}
\begin{proof}
By $\lambda=\lambda_s+\lambda_d\epsilon$ and $x=x_s+x_d\epsilon$ we have
\begin{eqnarray*}
\lambda x^{[m-1]}&=&\lambda x_s^{[m-1]}+\lambda(m-1)({\rm diag}((x_s)_1,\ldots,(x_s)_n))^{m-2}x_d\epsilon\\
&=&\lambda_sx_s^{[m-1]}+\epsilon\lambda_dx_s^{[m-1]}+\lambda_s(m-1)({\rm diag}((x_s)_1,\ldots,(x_s)_n))^{m-2}x_d\epsilon\\
&=&\lambda_sx_s^{[m-1]}+\epsilon(\lambda_dx_s^{[m-1]}+\lambda_s(m-1)({\rm diag}((x_s)_1,\ldots,(x_s)_n))^{m-2}x_d),
\end{eqnarray*}
where $(x_s)_i$ is the $i$-th component of $x_s$.

By $\mathcal{A}=\mathcal{A}_s+\mathcal{A}_d\epsilon$ we get
\begin{eqnarray*}
\mathcal{A}x^{m-1}=\mathcal{A}_sx^{m-1}+\epsilon\mathcal{A}_dx^{m-1}=\mathcal{A}_sx^{m-1}+\epsilon\mathcal{A}_dx_s^{m-1}.
\end{eqnarray*}
By computation, we have
\begin{eqnarray*}
\mathcal{A}_sx^{m-1}&=&\mathcal{A}_sx_s^{m-1}+\left(\sum_{k=2}^mA^{(k)}\right)x_d\epsilon,\\
\mathcal{A}x^{m-1}&=&\mathcal{A}_sx_s^{m-1}+\epsilon\left(\left(\sum_{k=2}^mA^{(k)}\right)x_d+\mathcal{A}_dx_s^{m-1}\right),
\end{eqnarray*}
where $A^{(k)}$ is an $n\times n$ matrix whose $(i,j)$-entry is ($1\leq i\leq j\leq n$)
\begin{eqnarray*}
(A^{(k)})_{ij}=\sum_{i_2,\ldots,i_{k-1},i_{k+1},\ldots,i_m=1}^n(\mathcal{A}_s)_{ii_2\cdots i_{k-1}ji_{k+1}\cdots i_m}(x_s)_{i_2}\cdots (x_s)_{i_{k-1}}(x_s)_{i_{k+1}}\cdots(x_s)_{i_m}.
\end{eqnarray*}
Hence $\mathcal{A}x^{m-1}=\lambda x^{[m-1]}$ if and only if
\begin{eqnarray*}
\mathcal{A}_sx_s^{m-1}+\epsilon\left(\left(\sum_{k=2}^mA^{(k)}\right)x_d+\mathcal{A}_dx_s^{m-1}\right)=\lambda_sx_s^{[m-1]}\\
+\epsilon(\lambda_dx_s^{[m-1]}+\lambda_s(m-1)({\rm diag}((x_s)_1,\ldots,(x_s)_n))^{m-2}x_d),
\end{eqnarray*}
that is, $\mathcal{A}_sx_s^{m-1}=\lambda_sx_s^{[m-1]}$ and
\begin{eqnarray*}
\left((m-1)\lambda_s({\rm diag}((x_s)_1,\ldots,(x_s)_n))^{m-2}-\sum_{k=2}^mA^{(k)}\right)x_d=(\mathcal{A}_d-\lambda_d\mathcal{I})x_s^{m-1}.
\end{eqnarray*}
Moreover, if $\mathcal{A}_s$ is symmetry, then $A^{(2)}=\cdots=A^{(m)}=A$, where $A$ is an $n\times n$ symmetry matrix with entries
\begin{eqnarray*}
(A)_{ij}=\sum_{i_3,i_4,\ldots,i_m=1}^n(\mathcal{A}_s)_{iji_3i_4\cdots i_m}(x_s)_{i_3}(x_s)_{i_4}\cdots(x_s)_{i_m}~(1\leq i\leq j\leq n).
\end{eqnarray*}
\end{proof}
Let ${\rm Ker}(M)=\{x:Mx=0,x\in\mathbb{R}^n\}$ denote the kernel of an $m\times n$ real matrix $M$. We now give the following spectral properties for a dual tensor with weakly irreducible nonnegative standard part.
\begin{thm}
Let $\mathcal{A}=\mathcal{A}_s+\mathcal{A}_d\epsilon$ be a dual tensor of order $m$ and dimension $n$, where $\mathcal{A}_s$ is a weakly irreducible nonnegative tensor. Then $\mathcal{A}$ has a positive dual number eigenvalue $\lambda=\lambda_s+\lambda_d\epsilon$ with a positive dual eigenvector $x=x_s+x_d\epsilon$. Furthermore, $\lambda_s$, $\lambda_d$, $x_s$ and $x_d$ have the following properties.\\
(1) $\lambda_s$ is the spectral radius of $\mathcal{A}_s$ and $\mathcal{A}_sx_s^{m-1}=\lambda_sx_s^{[m-1]}$.\\
(2) The matrix $M$ defined in Theorem 3.1 is an irreducible singular $M$-matrix such that
\begin{eqnarray*}
{\rm Ker}(M)=\{cx_s:c\in\mathbb{R}\},~\lambda_d=\frac{y^\top(\mathcal{A}_dx_s^{m-1})}{y^\top x_s^{[m-1]}},
\end{eqnarray*}
where $y$ is the unique positive unit vector $y$ such that $y^\top M=0$. Moreover, if $\mathcal{A}_s$ is symmetry, then $y=(x_s^\top x_s)^{-1}x_s$.\\
(3) Let $M^{(1)}$ be a $\{1\}$-inverse of $M$. Then
\begin{eqnarray*}
x_d\in S=\{M^{(1)}(\mathcal{A}_d-\lambda_d\mathcal{I})x_s^{m-1}+cx_s:c\in\mathbb{R}\}
\end{eqnarray*}
and for any $z\in S$, $x_s+z\epsilon$ is an eigenvector of the eigenvalue $\lambda$ of $\mathcal{A}$.\\
(4) If $x_s^\top x_d=0$, then $x_d=M^\#(\mathcal{A}_d-\lambda_d\mathcal{I})x_s^{m-1}$.
\end{thm}
\begin{proof}
Let $M$ be the matrix defined in Theorem 3.1. By Theorem 3.1, we know that $\mathcal{A}x^{m-1}=\lambda x^{[m-1]}$ if and only if $\mathcal{A}_sx_s^{m-1}=\lambda_sx_s^{[m-1]}$ and
\begin{equation}
Mx_d=(\mathcal{A}_d-\lambda_d\mathcal{I})x_s^{m-1}.\tag{3.1}
\end{equation}
Since $\mathcal{A}_s$ is a weakly irreducible nonnegative tensor, by Lemma \ref{lem2.4}, we can choose positive real number $\lambda_s$ and positive real vector $x_s$ such that $\mathcal{A}_sx_s^{m-1}=\lambda_sx_s^{[m-1]}$ and $\lambda_s$ is the spectral radius of $\mathcal{A}_s$. So we need to show that there are real number $\lambda_d$ and real vector $x_d$ satisfying (3.1).

Since $\mathcal{A}_s$ is a weakly irreducible nonnegative tensor and $x_s$ is positive, $M$ is an irreducible Z-matrix. By computation, we get $Mx_s=0$. From Lemmas \ref{lem2.2} and \ref{lem2.3}, we know that $M$ is an $n\times n$ singular $M$-matrix with ${\rm rank}(M)=n-1$. Then there exists an unique positive unit vector $y$ such that $y^\top M=0$. If $\lambda_d$ and $x_d$ satisfying (3.1), then $y^\top((\mathcal{A}_d-\lambda_d\mathcal{I})x_s^{m-1})=0$, that is, $(\mathcal{A}_d-\lambda_d\mathcal{I})x_s^{m-1}\in R(M)$. In this case, we get
\begin{eqnarray*}
\lambda_d=\frac{y^\top(\mathcal{A}_dx_s^{m-1})}{y^\top x_s^{[m-1]}}.
\end{eqnarray*}
Moreover, if $\mathcal{A}_s$ is symmetry, then $M$ is symmetry and $y=(x_s^\top x_s)^{-1}x_s$.

Notice that $Mx_s=0$ and ${\rm rank}(M)=n-1$. So we obtain
\begin{eqnarray*}
{\rm Ker}(M)=\{cx_s:c\in\mathbb{R}\}.
\end{eqnarray*}
When $\lambda_d=\frac{y^\top(\mathcal{A}_dx_s^{m-1})}{y^\top x_s^{[m-1]}}$, we have $(\mathcal{A}_d-\lambda_d\mathcal{I})x_s^{m-1}\in R(M)$ and
\begin{eqnarray*}
MM^{(1)}(\mathcal{A}_d-\lambda_d\mathcal{I})x_s^{m-1}=(\mathcal{A}_d-\lambda_d\mathcal{I})x_s^{m-1}.
\end{eqnarray*}
Then the general solution of linear equations $Mx_d=(\mathcal{A}_d-\lambda_d\mathcal{I})x_s^{m-1}$ is
\begin{eqnarray*}
x_d=M^{(1)}(\mathcal{A}_d-\lambda_d\mathcal{I})x_s^{m-1}+cx_s.
\end{eqnarray*}

From the above arguments, we know that $\mathcal{A}$ has a positive dual number eigenvalue $\lambda=\lambda_s+\lambda_d\epsilon$ with a positive dual number eigenvector $x=x_s+x_d\epsilon$, and parts (1),(2) and (3) hold.

By part (3), there is a constant $c$ such that
\begin{eqnarray*}
x_d=M^\#(\mathcal{A}_d-\lambda_d\mathcal{I})x_s^{m-1}+cx_s.
\end{eqnarray*}
Since $x_s^\top M^\#=0$, we have $x_d=M^\#(\mathcal{A}_d-\lambda_d\mathcal{I})x_s^{m-1}$ when $x_s^\top x_d=0$. So part (4) holds.
\end{proof}
Let $e_i$ denote the unit column vector whose $i$-th coordinate is $1$, and the other coordinates are zeros. We can obtain the following result from Theorem 3.2.
\begin{thm}
Let $\mathcal{A}=\mathcal{A}_s+\mathcal{A}_d\epsilon$ be a dual tensor of order $m$ and dimension $n$, where $\mathcal{A}_s$ is a weakly irreducible nonnegative tensor. Let $\lambda=\lambda_s+\lambda_d\epsilon$ be a positive dual number eigenvalue of $\mathcal{A}$ with a positive dual eigenvector $x=x_s+x_d\epsilon$. If $(x_s)_i=(x_s)_j$, then
\begin{eqnarray*}
(x_d)_i-(x_d)_j=(e_i^\top-e_j^\top)M^\#(\mathcal{A}_d-\lambda_d\mathcal{I})x_s^{m-1}=(e_i^\top-e_j^\top)M^{(1)}(\mathcal{A}_d-\lambda_d\mathcal{I})x_s^{m-1}
\end{eqnarray*}
for any $\{1\}$-inverse $M^{(1)}$ of $M$, where $M$ is the matrix defined in Theorem 3.1.
\end{thm}
\begin{proof}
If $(x_s)_i=(x_s)_j$, then $(e_i^\top-e_j^\top)x_s=0$. From part (2) of Theorem 3.2, we know that $e_i-e_j\in R(M^\top)$. Theorem 3.1 implies that $(\mathcal{A}_d-\lambda_d\mathcal{I})x_s^{m-1}\in R(M)$. By Lemma \ref{lem2.1}, we obtain
\begin{eqnarray*}
(e_i^\top-e_j^\top)M^\#(\mathcal{A}_d-\lambda_d\mathcal{I})x_s^{m-1}=(e_i^\top-e_j^\top)M^{(1)}(\mathcal{A}_d-\lambda_d\mathcal{I})x_s^{m-1}
\end{eqnarray*}
for any $\{1\}$-inverse $M^{(1)}$ of $M$. By part (3) of Theorem 3.2, we have
\begin{eqnarray*}
(x_d)_i-(x_d)_j=(e_i^\top-e_j^\top)M^\#(\mathcal{A}_d-\lambda_d\mathcal{I})x_s^{m-1}.
\end{eqnarray*}
\end{proof}
\section{Dual eigenvalues and eigenvectors of hypergraphs}
Let $H$ be an $n$-vertex $m$-uniform connected hypergraph. The \textit{adjacency tensor} of $H$, denoted by $\mathcal{A}_H$, is an order $m$ dimension $|V(H)|$ tensor with entries
\begin{eqnarray*}
a_{i_1i_2\cdots i_m}=\begin{cases}\frac{1}{(m-1)!}~~~~~~~\mbox{if}~i_1i_2\cdots i_m\in E(H),\\
0~~~~~~~~~~~~~\mbox{otherwise}.\end{cases}
\end{eqnarray*}
If $H$ is connected, then $\mathcal{A}_H$ is a weakly irreducible nonnegative tensor. In this case, $\mathcal{A}_H$ has a positive real eigenvalue $\lambda_s$ (the spectral radius of $H$) with a positive real eigenvector $x_s$.

For a vertex subset $U$, let $x_s^U=\prod_{i\in U}(x_s)_i$. Let $M(H,x_s)$ denote the $n\times n$ symmetry matrix with entries
\begin{eqnarray*}
(M(H,x_s))_{uv}=\begin{cases}(m-1)\lambda_s(x_s)_u^{m-2}~~~~~~~~~~~~\mbox{if}~u=v,\\
\sum_{\{u,v\}\subseteq e\in E(H)}x_s^{e\backslash\{u,v\}}~~~~~~~\mbox{if}~u\sim v,\\
0~~~~~~~~~~~~~~~~~~~~~~~~~~~~~~~~~\mbox{otherwise},\end{cases}
\end{eqnarray*}
where $u\sim v$ means there exists an edge $e\in E(H)$ containing two distinct vertices $u$ and $v$.

We can derive the following result from Theorem 3.2.
\begin{thm}
Let $H$ be an $m$-uniform connected hypergraph with an associated dual tensor $\mathcal{A}=\mathcal{A}_s+\mathcal{A}_d\epsilon$, where $\mathcal{A}_s$ is the adjacency tensor of $H$. Then $\mathcal{A}$ has a positive dual number eigenvalue $\lambda=\lambda_s+\lambda_d\epsilon$ with a positive dual eigenvector $x=x_s+x_d\epsilon$. Furthermore, $\lambda_s$, $\lambda_d$, $x_s$ and $x_d$ have the following properties.\\
(1) $\lambda_s$ is the spectral radius of $\mathcal{A}_s$ and $\mathcal{A}_sx_s^{m-1}=\lambda_sx_s^{[m-1]}$, $\lambda_d=\frac{x_s^\top(\mathcal{A}_dx_s^{m-1})}{\|x_s\|_m^m}$.\\
(2) $M(H,x_s)$ is an irreducible singular $M$-matrix such that
\begin{eqnarray*}
{\rm Ker}(M(H,x_s))=\{cx_s:c\in\mathbb{R}\}.
\end{eqnarray*}
(3) Let $M^{(1)}$ be a $\{1\}$-inverse of $M(H,x_s)$. Then
\begin{eqnarray*}
x_d\in S=\{M^{(1)}(\mathcal{A}_d-\lambda_d\mathcal{I})x_s^{m-1}+cx_s:c\in\mathbb{R}\}
\end{eqnarray*}
and for any $z\in S$, $x_s+z\epsilon$ is an eigenvector of the eigenvalue $\lambda$ of $\mathcal{A}$.\\
(4) If $x_s^\top x_d=0$, then $x_d=M(H,x_s)^\#(\mathcal{A}_d-\lambda_d\mathcal{I})x_s^{m-1}$.
\end{thm}
Let $x=x_s+x_d\epsilon$ be the vector in Theorem 4.1. If $x_s^\top x_d=0$ and $\|x_s\|_m=1$, then $x_s$ and $x_d$ are unique. In this case, the unique eigenvector $x$ is defined as the \textit{centrality vector} of the uniform hypergraph $H$. We can obtain the following result from Theorem 3.3.
\begin{thm}
Let $x=x_s+x_d\epsilon$ be the centrality vector of an $m$-uniform connected hypergraph $H$. If $(x_s)_i=(x_s)_j$ for two vertices $i$ and $j$, then
\begin{eqnarray*}
(x_d)_i-(x_d)_j&=&(e_i^\top-e_j^\top)M(H,x_s)^\#(\mathcal{A}_d-\lambda_d\mathcal{I})x_s^{m-1}\\
&=&(e_i^\top-e_j^\top)M^{(1)}(\mathcal{A}_d-\lambda_d\mathcal{I})x_s^{m-1}
\end{eqnarray*}
for any $\{1\}$-inverse $M^{(1)}$ of $M(H,x_s)$.
\end{thm}
By Theorem 4.1, the centrality vector of an uniform hypergraph $H$ can be obtained via the following steps.

Step 1. Compute the spectral radius $\lambda_s$ and Perron vector $x_s$ of $H$ ($\|x_s\|_m=1$).

Step 2. Use $x_s$ to compute $\lambda_d=x_s^\top(\mathcal{A}_dx_s^{m-1})$ and $M(H,x_s)$.

Step 3. Set $x_d=M(H,x_s)^\#(\mathcal{A}_d-\lambda_d\mathcal{I})x_s^{m-1}$.

Step 4. Output $x=x_s+x_d\epsilon$.

\section{Dual Centrality of vertices in graphs and hypergraphs}

By Theorem 4.1, we have for an $m$-uniform connected hypergraph ($m \geq 2$), when we add a perturb $\mathcal{A}_d$ to the adjacency tensor $\mathcal{A}_s$, the eigenvector centrality scores $x_s$ remain unchanged. (When $m = 2$, the above scenario reduces to an ordinary graph.)
We therefore ask whether comparing $x_d$ can distinguish vertices whose eigenvector centrality scores are equal. In practice, perturbations can be applied flexibly to hyperedges of special structures, such as triangles, $4$-cycles, stars, or even cut-edges .etc.

\begin{figure}[H]
\centerline{\includegraphics[scale=0.35]{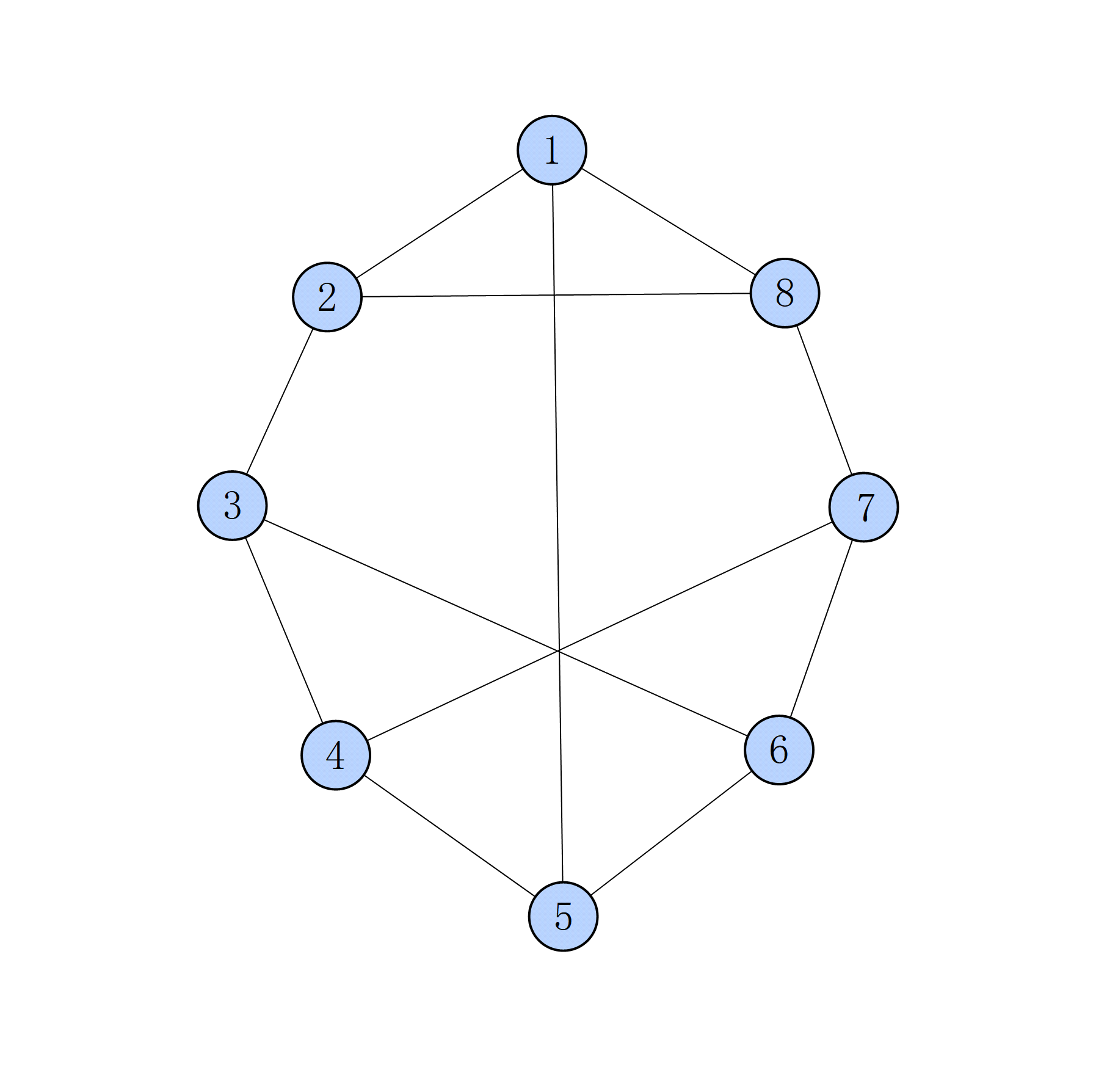}}
\caption{A regular garph with $8$ vertices.}
\label{fig1}
\end{figure}

Below we give an example to illustrate this idea.
The graph in Fig \ref{fig1} shows a regular graph with $8$ vertices, whose adjacency matrix is denoted by $A_s$.
The eigenvector centrality of every vertex is given by the perron vector $x_s$ of $A_s$,
and all entries of $x_s$ are equal.
We now introduce a dual perturbation on the single triangle of the graph, we set the entries $(1, 2),(2,1) (2, 8),(8,2),(1,8)$ and $(8, 1)$ of $A_d$ to $1$ and all other entries to $0$.
We wish to determine the ranking of the vertices with respect to the dual component vector $x_d$ of the dual matrix $A = A_s + A_d \epsilon$. The resulting vector is listed in the table below.

\begin{table*}[h]
    \tiny
    \centering
    \caption{The centrality scores of vertices in Fig \ref{fig1}.}
    \label{table1}
    \begin{tabular}{lcccccccc}  
    \toprule
    \textbf{Vertices} & \textbf{$1$} & \textbf{$2$} & \textbf{$3$} & \textbf{$4$} & \textbf{$5$} & \textbf{$6$}& \textbf{$7$} &\textbf{$8$}\\ \midrule
    \textbf{$\boldsymbol{x_s}$}  & 0.3536 & 0.3536 & 0.3536 & 0.3536 &0.3536 & 0.3536 & 0.3536  &0.3536  \\
    \textbf{$\boldsymbol{x_d}$} & \textbf{0.2983} & \textbf{0.2983} & -0.1436 & -0.2320 & -0.1436 & -0.2320 & -0.1436  & \textbf{0.2983} \\
    \bottomrule
    \end{tabular}
\end{table*}

From Table \ref{table1}, we have the ranking of vertices under $x_d$ is $1 = 2 = 8 > 3 =5 =7 > 4 = 6$.
Thus, once the triangle is dually perturbed the vertices are no longer tied in the $x_d$ ranking.
The three vertices that form the triangle occupy the top positions, followed by their immediate neighbours, while the vertices not adjacent to the triangle are ranked lowest.
%
%
\begin{figure}[H]
\centerline{\includegraphics[scale=0.35]{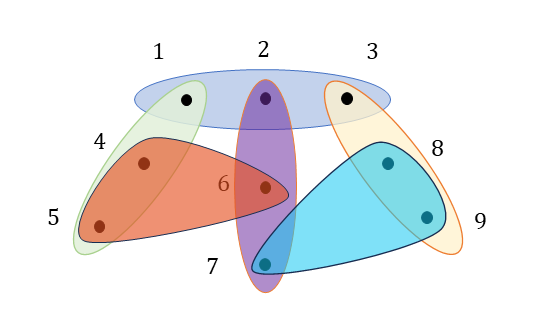}}
\caption{A regular hypergarph  with $9$ vertices.}
\label{fig2}
\end{figure}

The Figure \ref{fig2} shows a hypergraph with $9$ vertices, whose adjacency tensor is denoted by $\mathcal{A}_s$.
The eigenvector centrality of every vertex is given by the perron vector $x_s$ of $\mathcal{A}_s$,
and some entries of $x_s$ are equal.
We now introduce a dual perturbation on $\{1, 2, 3\}$, we set the entries $(1,2,3),(1,3,2),(2,1,3),(2,3,1),(3,1,2)$ and $(3,2,1)$ of $\mathcal{A}_d$ to $1$ and all other entries to $0$.
We wish to determine the ranking of the vertices with respect to the dual component vector $x_d$ of the dual tensor $\mathcal{A}= \mathcal{A}_s + \mathcal{A}_d \epsilon$. The resulting vector is listed in the table below.

\begin{table*}[h]
    \tiny
    \centering
    \caption{The centrality scores of vertices in Fig \ref{fig2}.}
    \label{table3}
    \begin{tabular}{lccccccccc}  
    \toprule
    \textbf{Vertices} & \textbf{$1$} & \textbf{$2$} & \textbf{$3$} & \textbf{$4$} & \textbf{$5$} & \textbf{$6$}& \textbf{$7$}& \textbf{$8$}& \textbf{$9$} \\ \midrule
    \textbf{$\boldsymbol{x_s}$}  & 0.4807 & 0.4807 & 0.4807 & 0.4807 &0.4807 & 0.4807 & 0.4807 & 0.4807 & 0.4807 \\
    \textbf{$\boldsymbol{x_d}$} & \textbf{0.2137} & \textbf{0.2137} &\textbf{0.2137}  & -0.1068 & -0.1068 & -0.1068 & -0.1068 &-0.1068 & -0.1068\\
    \bottomrule
    \end{tabular}
\end{table*}

As shown in Table \ref{table3}, the ranking of vertices under $x_s$ is $1 =2 = \cdots = 9 $.
But under $x_d$ , we observe that the ranking changes to $1 =2=3 >4=5=6=7=8=9$.
We conclude that when a hyperedge $e_1$ is perturbed, and vertices $i \in e_1$, $j \notin e_1$ with $(x_s)_i=(x_s)_j$, then $(x_d)_i>(x_d)_j$.
We would like to know whether the resulting ranking remains the same when different hyperedges are perturbed.
We set the entries $(4,5,6),(4,6,5),(5,4,6),(5,6,4),(6,4,5)$ and $(6,5,4)$ of $\mathcal{A}_d$ to $1$ and all other entries to $0$.The resulting vector is listed in the table below.

\begin{table*}[h]
    \tiny
    \centering
    \caption{The centrality scores of vertices in Fig \ref{fig2}.}
    \label{table4}
    \begin{tabular}{lccccccccc}  
    \toprule
    \textbf{Vertices} & \textbf{$1$} & \textbf{$2$} & \textbf{$3$} & \textbf{$4$} & \textbf{$5$} & \textbf{$6$}& \textbf{$7$}& \textbf{$8$}& \textbf{$9$} \\ \midrule
    \textbf{$\boldsymbol{x_s}$}  & 0.4807 & 0.4807 & 0.4807 & 0.4807 &0.4807 & 0.4807 & 0.4807 & 0.4807 & 0.4807 \\
    \textbf{$\boldsymbol{x_d}$} & 0.0855 & -0.1068 &-0.2991 & \textbf{0.5342} & \textbf{0.5342} & 0.3419 & -0.2350 &-0.4273 & -0.4273\\
    \bottomrule
    \end{tabular}
\end{table*}

From Table \ref{table4}, we have the ranking of vertices under $x_d$ is $4=5>6>1>2>7>3>8=9$.
Vertices $4$ and $5$ are ranked higher than vertex $6$, because in addition to belonging to hyperedge $\{4, 5, 6\}$, they also appear in hyperedge $\{1, 4, 5\}$.
Vertices $1$, $2$, and $7$ are adjacent to all of the important vertices ($4$, $5$, and $6$). In contrast, vertices $3$, $8$, and $9$ are only connected to peripheral vertices. Therefore, vertices $1$, $2$, and $7$ ranked higher.

\section*{References}
\bibliographystyle{plain}
\bibliography{ml0ht2}
\end{spacing}
\end{document}